\documentclass{article}
\usepackage{amssymb}
\usepackage{amsmath}
\usepackage{harvard}
\usepackage{appendix}
\usepackage{babel}
\usepackage{amsfonts}
\usepackage{cite}
\usepackage{graphicx}
\usepackage{subcaption}

\newtheorem{theorem}{Theorem}

\newtheorem{example}[theorem]{Example}

\newtheorem{notation}[theorem]{Notation}

\newtheorem{proposition}[theorem]{Proposition}

\newenvironment{proof}[1][Proof]{\noindent\textbf{#1.} }{\ \rule{0.5em}{0.5em}}

\begin{document}

\title{Dynamical properties of a cellular automaton model of language shift
in Algeria }
\author{{\small \ AIT SADI Nassima}$^{1}$ \ \ \ \ \ \ \ \ {\small CHEMLAL
Rezki}$^{2}$ \\
$^{1,2}$Laboratory of applied mathematics, Bejaia University, Algeria}
\maketitle

\begin{abstract}
In this paper we are interested in the dynamical properties of a two
dimensional cellular automaton that model the language shift in Algeria.
We will study among other properties the surjectivity and the existence of equicontinuity points for the cellular automaton. 
\newline
Using computer simulation we obtained some properties of the bassin of attraction of some fixed points. This experimental results are provided along with analytical proofs.

\begin{description}
\item[Keywords:] Cellular automata, topological dynamics.
\end{description}
\end{abstract}

\section{Introduction}

Cellular automata are employed in variety of modeling contexts. They dates
back to Von Neumann in the late $1940$s and the first model were
biologically motivated. During the last years they have been used to model
complex systems in many different domains, for example in physics, biology
and social choices \cite{physics,social}.

A Cellular automaton is a discrete model of computation that evolve in space and time. It is consist of a regular grid of cells that can adopt only
one given state at each time unit. The state of each cell is determined by applying the associated local rule of the cellular automaton to the current state of the cell and the states of the cells of its neighborhood. A
configuration is a snapshot of the state of all automata in the lattice. The
lattice is usually $%
\mathbb{Z}
^{n}$ where $n$ represent the dimension of the cellular automaton.

One way to study the complexity of a cellular automaton is to endow the
space of configurations with the product topology and consider the cellular automaton
as a discrete dynamical system.

In this paper we focus on studying the dynamical properties of a $2-$dimensional cellular automaton that model the language shift in Algeria \cite%
{Chemlal}.
\newline
In first part we study the existence of periodic points, the existence of blocking
words and the surjectivity of the cellular automaton.
\newline
In the second part we
rely on a computer simulation to deduce and show some properties of the
bassin of attraction of the fixed point $3^{\infty \times \infty }$ for different values of the parameters $P_z$ and $P_e$.
\section{Basic definitions}

Let $A$ be an alphabet. A pattern $w$ is a set of values from $A$ on finite
connected subset of coordinates $E\subset 
\mathbb{Z}
^{2}.$ The length of the pattern $w\in A^{n\times m}$ is \newline
$|w|=(|w|_{1},|w|_{2})=(n,m).$ We denote by $A^{%
\mathbb{Z}
^{2}}$ the set of all configurations in $%
\mathbb{Z}
^{2}$ constructed over the alphabet $A$.%
\begin{equation*}
\begin{tabular}{ccccc}
&  & $\vdots $ &  &  \\ 
& $x_{(-1,1)}$ & $x_{(0,1)}$ & $x_{(2,0)}$ &  \\ 
$\cdots $ & $x_{(-1,1)}$ & $x_{(0,0)}$ & $x_{(1,0)}$ & $\cdots $ \\ 
& $x_{(-1,1)}$ & $x_{(0,-1)}$ & $x_{(1,-1)}$ &  \\ 
&  & $\vdots $ &  & 
\end{tabular}%
\end{equation*}%
For any pattern $w$ we define the cylinder $[w]$ at position $(i,j)$ by:%
\begin{equation*}
\lbrack w]_{(i,j)}=\{x\in A^{%
\mathbb{Z}
^{2}}:x_{[i,i+|w|_{1}]\times \lbrack j,j+|w|_{2}]}=w\}
\end{equation*}%
Endowed with the following distance%
\begin{equation*}
\forall x,y\in A^{%
\mathbb{Z}
^{2}},d(x,y)=2^{-\min \{||\overrightarrow{i}||_{\infty }:x_{i}\neq
y_{i},i=(i_{1},i_{2})\in 
\mathbb{Z}
^{2}\}}\text{where }||\overrightarrow{i}||_{\infty }=\underset{j=\overline{%
1,2}}{\max }\{i_{j}\}
\end{equation*}%
the space $A^{%
\mathbb{Z}
^{2}}$ is compact, perfect, and totally disconnected in the product
topology. Any vector $\overrightarrow{i}\in 
\mathbb{Z}
^{2}$ determines a continuous shift map\newline
$\sigma _{\overrightarrow{i}}:A^{%
\mathbb{Z}
^{2}}\rightarrow A^{%
\mathbb{Z}
^{2}}$ defined by $\sigma _{\overrightarrow{i}}(x)_{\overrightarrow{j}}=x_{%
\overrightarrow{j}+\overrightarrow{i}},\forall \overrightarrow{j}\in 
\mathbb{Z}
^{2}.$ The set $\{\sigma _{\overrightarrow{i}}\}_{\overrightarrow{i}\in 
\mathbb{Z}
^{2}}$ is denoted simply by $\sigma .$

\subsection{Cellular automata as dynamical systems:}

Let $E$ be a subset in $%
\mathbb{Z}
^{2}$ and $f:A^{\Bbbk }\rightarrow A$ be a function called \emph{the local
rule}. \newline
A cellular automaton determined by $f$ is the function $F:A^{%
\mathbb{Z}
^{2}}\rightarrow A^{%
\mathbb{Z}
^{2}}$ defined by $F(x)_{m}=f(x_{m+\Bbbk })$ for all $x\in A^{%
\mathbb{Z}
^{2}}$ and all $m\in 
\mathbb{Z}
^{2}.$ Curtis-Hedlund and Lyndon \cite{Hedlund} showed that cellular
automata are exactly the continuous transformations of $A^{%
\mathbb{Z}
^{2}}$ that commute with all shifts. We refer to $\Bbbk $ as a neighborhood
of $F.$ Generally, we use the Von Neumann neighborhood or the Moore
neighborhood in dimension $2.$%

\begin{figure}[h]
    \centering
    \begin{minipage}{0.49\textwidth}
    \centering
    \includegraphics{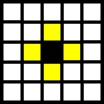}
     \caption{Von Neumann}
     \label{fig:Neumann neighboor}  
     \end{minipage}
     \hfill
    \begin{minipage}{0.49\textwidth}
    \centering
    \includegraphics{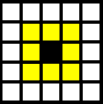}
    \caption{Moore}
    \label{fig:Moore neighboor}
    \end{minipage}
\end{figure}

\newpage
A point $r\in A^{%
\mathbb{Z}
^{2}}$ is fixed point if $F(r)=r$ and his bassin of attraction is the set $%
B(r)=\{x\in A^{%
\mathbb{Z}
^{2}}:\underset{n\rightarrow +\infty }{\lim }F^{n}(x)=r\}.$

A point $x$ is a periodic point of period $p$ if $F^{p}(x)=x$ and $\forall
i<p,$ $F^{i}(x)\neq x.$ If there exist an integer $m$ such that $F^{m}(x)$
is periodic then $x$ is called an ultimately periodic point.

Denote the periodic points of the shift by specially periodic points and the
points which are periodic for $F$ without being periodic for the shift by
strictly temporally periodic points \cite{periodic}.

A pattern $w$ of size $k\times l$ is called $(r,s)-$blocking with offset $%
(p,q)$ if there exist non-negative integers $p\leq k-r$ and $q\leq l-s$ such
that for all $x,y\in \lbrack w]_{(0,0)}$ and all $n\geq 0$, we have: $%
F^{n}(x)_{[p,p+r)\times \lbrack q,q+s)}=F^{n}(y)_{[p,p+r)\times \lbrack
q,q+s)}$.\newline
When the offset $(p,q)=(0,0)$ the pattern $w$ is said fully blocking.

A point $x$ is a point of equicontinuity for the cellular automaton $(A^{%
\mathbb{Z}
^{2}},F)$ if 
\begin{equation*}
\forall \varepsilon >0,\exists \delta >0:\forall y,d(x,y)<\delta \Rightarrow
\forall n\in 
\mathbb{N}
,d(F^{n}(x),F^{n}(y))<\varepsilon
\end{equation*}%
A cellular automaton $F$ is equicontinuous if all points are points of
equicontinuity and is almost equicontinuous if the set of equicontinuity
contain a countable intersection of dense open sets.\newline
$F$ is sensitive to the initial conditions if%
\begin{equation*}
\exists \varepsilon >0,\forall x\in A^{%
\mathbb{Z}
^{2}},\forall \delta >0,\exists y:d(x,y)<\delta ,\exists n\in 
\mathbb{N}
:d(F^{n}(x),F^{n}(y))\geq \varepsilon
\end{equation*}

Kurka \cite{Kurka} introduced a classification of one dimensional cellular
automata according to the equicontinuity and sensitivity.

\begin{theorem}
Let $F$ be a one dimensional cellular automaton with radius $r.$ The
following properties are equivalent:

\begin{enumerate}
\item $F$ is not sensitive.

\item $F$ admit an $r-blocking$ word.

\item $F$ is almost equicontinuous.
\end{enumerate}
\end{theorem}

This equivalence do not hold in higher dimension, Gamber \cite{Emilie}
proved that the implication $3\Rightarrow 1$ and $1\Rightarrow 2$ are
valaible in all dimensions and the existence of a fully blocking pattern
implies the almost equicontinuity of the cellular automaton.

\begin{theorem}[\protect\cite{Emilie}]
Let $F:A^{%
\mathbb{Z}
^{D}}\rightarrow A^{%
\mathbb{Z}
^{D}}$ with radius $r.$ If there exists a fully blocking pattern of size $%
k\times k\times ...\times k$ for $F$ where $k\geq r$ then $F$ is almost
equicontinuous.
\end{theorem}

\section{A cellular automaton model for language shift in Algeria}

In \cite{Chemlal}, the author presented a cellular automaton model of the
language shift in Algeria where the two dominants languages are the Algerian
Arabic and the English. An individual can shift from Amazigh to Arabic and/or from French to English states.
\newline
Four states are
possible:

\begin{enumerate}
\item State $^{\prime }00^{\prime }$: Person speak Amazigh and use
French as second language.

\item State $^{\prime }01^{\prime }$: Person speak  Amazigh and use
English as second language.

\item State $^{\prime }10^{\prime }$: Person speak Arabic and use French
as second language.

\item State $^{\prime }11^{\prime }$: Person speak Arabic and use English as
second language.
\end{enumerate}

Each state $(mn)$ is coded by an integer $p$ such that $p=n\times
2^{0}+m\times 2^{1}$.%
 \begin{table}[h]
     \centering
    \begin{tabular}{|c|c|c|c|}
     \hline
     $00$  &  $01$   &   $10$  &  $11$ \\
     \hline
     $0$ & $1$ & $2$ & $3$ \\
     \hline 
    \end{tabular}
    \caption{Encoding the four states}
    \label{tab:my_label}
\end{table}

\begin{notation}
Let $x$ be a point in $A^{%
\mathbb{Z}
^{2}}.$

\begin{enumerate}
\item $\sum\nolimits_{0_{i,j}}=\sum\limits_{r=-1}^{1}\sum\limits_{s=-1}^{1}%
\left[ \frac{x_{i+r,j+s}}{2}\right] $ where $\left[ {}\right] $ is the
integer part.

\item $\sum\nolimits_{1_{i,j}}=\sum\limits_{r=-1}^{1}\sum%
\limits_{s=-1}^{1}(x_{i+r,j+s}-2.\left[ \frac{x_{i+r,j+s}}{2}\right] )$
where $\left[ {}\right] $ is the integer part.

\item $0\leq P_{z}\leq 9$ and $0\leq P_{e}\leq 9$ are the parameters
defining the pressure toward maintaining the Amazigh and the Algerian Arabic
respectively.
\end{enumerate}
\end{notation}
\newpage

Consider the alphabet $A=\{0,1,2,3\}$. The local rule of the cellular
automaton $K:A^{%
\mathbb{Z}
^{2}}\rightarrow A^{%
\mathbb{Z}
^{2}}$ that model the language shift in Algeria is summarized in the
following figure:
\begin{figure}[h]
    \centering
    \includegraphics[width=\textwidth]{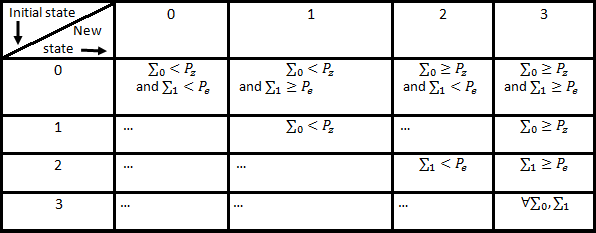}
    \caption{The local rule}
    \label{fig:my_label}
\end{figure}

\subsection{Dynamical properties}

In this section, we give some dynamical properties of the cellular
automaton $K$ seen as a dynamical system.

\begin{proposition}
For $(P_{z},P_{e})=(9,9)$ all points from $(A^{%
\mathbb{Z}
^{2}},K)$ are fixed points.
\end{proposition}

\begin{proof}
\ \ \ \ \ \ \newline
Suppose that $(P_{z},P_{e})=(9,9)$ and take a point $y\in $ $A^{%
\mathbb{Z}
^{2}}$. Let $(i,j)$ $\in 
\mathbb{Z}
^{2}$.\newline
If $y_{(i,j)}=0$ then $0\leq \sum_{0}(y_{[-i,i]\times \lbrack -i,i]})\leq 8$
and $0\leq \sum_{1}(y_{[-i,i]\times \lbrack -i,i]})\leq 8$. 
So, $%
K(y)_{(i,j)}=0.$\newline
If $y_{(i,j)}=1$ then $0\leq \sum_{0}(y_{[-i,i]\times \lbrack -i,i]})\leq 8$%
. So, $K(y)_{(i,j)}=1$\newline
If $y_{(i,j)}=2$ then $0\leq \sum_{1}(y_{[-i,i]\times \lbrack -i,i]})\leq 8$%
.So, $K(y)_{(i,j)}=2.$\newline
The case when $y_{(i,j)}=3$ is trivial.
\end{proof}

\begin{proposition}
For any values of the parameters $P_{z},P_{e}$ different from $(9,9),$ the
cellular automaton $K$ is not surjective.
\end{proposition}

\begin{proof}
Consider $P_{z}=0$ or $P_{e}=0$. For all $x\in A^{%
\mathbb{Z}
^{2}}$ and all $k\in 
\mathbb{Z}
$ we have: \newline
$\sum_{0}(x_{[k-1,k+1]\times \lbrack k-1,k+1]})\geq 0$ or $%
\sum_{1}(x_{[k-1,k+1]\times \lbrack k-1,k+1]})\geq 0$ what means that $%
K(x)_{k}\in \{2,3\}$ or $K(x)_{k}\in \{1,3\}$ then the point $0^{\infty
\times \infty }$ do not have an antecedent. Now,Let take the point $y\in A^{%
\mathbb{Z}
^{2}}$ defined as follow:%
\begin{equation*}
\left\{ 
\begin{array}{c}
y_{00}=0 \\ 
\forall i\in 
\mathbb{Z}
^{2}:\left\Vert \overrightarrow{i}\right\Vert _{\infty }=1,y_{%
\overrightarrow{i}}=3 \\ 
0\text{ otherwise}%
\end{array}%
\right.
\end{equation*}%
Assume by contradiction that there is a point $x\in A^{%
\mathbb{Z}
^{2}}$ such that $K(x)=y$ then by construction, the letter $0$ can either be
turned to an other letter from $A$ different from $0$ or be fixed. So we
have $x_{00}=0$ and%
\begin{equation*}
\forall i\in 
\mathbb{Z}
^{2}:\left\Vert \overrightarrow{i}\right\Vert _{\infty }>1,x_{%
\overrightarrow{i}}=0
\end{equation*}%
In the following, we show that for any value of $P_{z}$ in $\{0,...,8\}$ the
point $y$ do not have an a preimage by $K$ independently from the value of $%
P_{e}.$

\begin{itemize}
\item If $P_{z}=1$ then to keep $x_{00}=0$ the elements of $x$ at the
positions $i\in 
\mathbb{Z}
^{2}$ such that $||\overrightarrow{i}||=1$ have to belong in $\{0,1\}$ and
whatever the number of $1^{\prime }s$ that $x$ contains on this positions we
will get $K(x)\neq $ $y$.

\item If $P_{z}=2$ then the elements of $x$ at the rest of the undetermined
positions must have at most one $2$ or one $3$ and all possibilities give $%
K(x)\neq y.$

\item If $P_{z}=3$ then at most there exist two integers $i_{1},i_{2}\in 
\mathbb{Z}
^{2}$ such that $\left\Vert \overrightarrow{i_{j}}\right\Vert _{\infty
}=1,\forall j\in \{1,2\}$ and $x_{i_{j}}\in \{2,3\}$. Similarly we get $%
K(x)\neq y$.

\item If $4\leq P_{z}\leq 8$ then in order to have $K(x)_{||\overrightarrow{i%
}||_{\infty }=1}=3$, $x$ must have a $3$ at positions $i\in 
\mathbb{Z}
^{2}$ such that$||\overrightarrow{i}||_{\infty }=1$. This is a
contradiction with the fact that $\sum_{0}(x_{[-1,1]\times \lbrack
-1,1]})<P_{z}$.\newline

Using same arguments one can show that for $P_{z}=9$ and any value of $%
P_{e}$ from $\{0,...,8\}$ the same point $y$ do not have a preimage.
\end{itemize}
\end{proof}

\begin{proposition}
For all parameters $0\leq P_{z},P_{e}\leq 9$, $K$ do not have any strictly
periodic point.
\end{proposition}

\begin{proof}
Use contradiction and suppose that there exist a point $x\in A^{%
\mathbb{Z}
^{2}}$ which is strictly periodic for $K$, i.e, $\exists p\in 
\mathbb{N}
:K^{p}(x)=x$ and $\forall i<p$ $K^{i}(x)\neq x.$\newline
As: $\forall i<p$ $K^{i}(x)\neq x$, without loss of generalities $\exists
k\in 
\mathbb{Z}
$ st. $K^{i}(x)_{kk}\neq x_{kk}.$\newline
Or, by construction the letters $0,1,2,3$ turn by following the schema in
figure \ref{1} and this is a contradiction.
\begin{figure}[h]
    \centering
    \includegraphics[width=4.5cm]{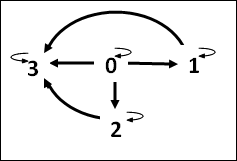}
    \caption{Letters behavior}
    \label{1}
\end{figure}
\end{proof}

\begin{proposition}
For all parameters $0\leq P_{z},P_{e}\leq 9$, $k$ is almost equicontinuous.
\end{proposition}

\begin{proof}
By definition of the local rule of $K$, $3$ is a fully $(1,1)$-blocking word
for any values of the parameters $P_{z}$ and $P_{e}$. It follow from the
theorem $2$ that $K$ is almost equicontinuous.
\end{proof}

\subsection{Bassin of attraction of fixed points}

In this section we will focus on the bassin of attraction of fixed points.%
\newline
The cellular automaton $K$ admits an infinity of fixed points. Excepting the point $3^{\infty \times \infty
},$ the bassin of attraction of each fixed point do not contain any cylinder for all values of the
parameters $P_{z}$ and $P_{e}$. It has been proven in 
\cite{Chemlal} that the bassin of attraction of $3^{\infty \times \infty }$
contain the cylinder $[3]$ for $(P_{z},P_{e})=(9,9)$ and do not contain any
cylinder for $(P_{z},P_{e})=(1,1).$ In the following we investigate the
bassin of attraction of this point for other values of $P_{z}$ and $P_{e}$
differents from $(1,1)$ and $(9,9).$

\subsubsection{Computer simulation}

In order to characterize the bassin of attraction of a fixed point for some
values of $P_{z}$ and $P_{e}.$ We have a program that receives any pattern $%
P $ of any size $n\times m$ and a fixed values of $P_{z},P_{e}$ and returns
a list of all the possible patterns $P^{\prime }$ of size $(n+2)\times (m+2)$
such that $K(P^{\prime })=P$.

As the cellular automaton $K$ is not surjective then there is some cases where the program returns an empty list. Inside the program we needed to generalize 
all the possible $3\times 3$ patterns that we can form by
the numbers $0,1,2$ and $3$. To do that we used a command that returns those patterns as a set, what means that for example instead to
get in total $9$ patterns that contain one $1$ and $0^{\prime }s$ otherwise
we had just one pattern. That do not affect our results because it is enough
to shift the one obtained pattern in all directions. 

As the bassin of attraction of a fixed point consist in the set of points that
converges under the iterations of the cellular automaton $K$ to the fixed
point. It may be infinite then our program cannot caracterise the whole bassin but give enough information to be able to deduce its properties. 

The idea is that the program receives a pattern $P$ that belong to $%
3^{\infty \times \infty }$ and a particular values of $P_{z}$ and $P_{e}$
and give the list of preimages of $P$ when it is possible. After that some elements are randomly chosen and the process is repeated again.

\begin{example}
Consider $P=\left( 
\begin{array}{cc}
3 & 3 \\ 
3 & 3%
\end{array}%
\right) $ and $(P_{z},P_{e})=(2,2)$.\newline
We take randomly one pattern that belong to the list of all possible preimages of $K^{-1}(P)$ and re-execute the program with the new taken pattern. 
We repeat the process many time as we need. If never at step $k+1$ we get back an empty list then we chose an other pattern from the list of $K^{-k}(P)$ until either we get a pattern with preimages or we try all the obtained pattern from the precedent list. In the following we present some of the pattern that we had obtained:
\begin{figure}[h]
    \centering
    \begin{subfigure}[b]{0.3\textwidth}
    \centering
    \includegraphics[width=3cm]{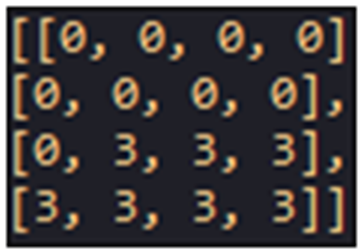}
     \caption{$K^{-1}(P)$}
     \label{fig:Neumann neighboor}  
     \end{subfigure}
     \hfill
    \begin{subfigure}[b]{0.3\textwidth}
    \centering
    \includegraphics[width=4cm]{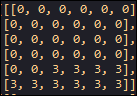}
    \caption{$K^{-2}(P)$}
    \label{fig:Moore neighboor}
    \end{subfigure}
    \hfill
    \begin{subfigure}[b]{0.3\textwidth}
    \centering
    \includegraphics[width=5cm,height=3.5cm]{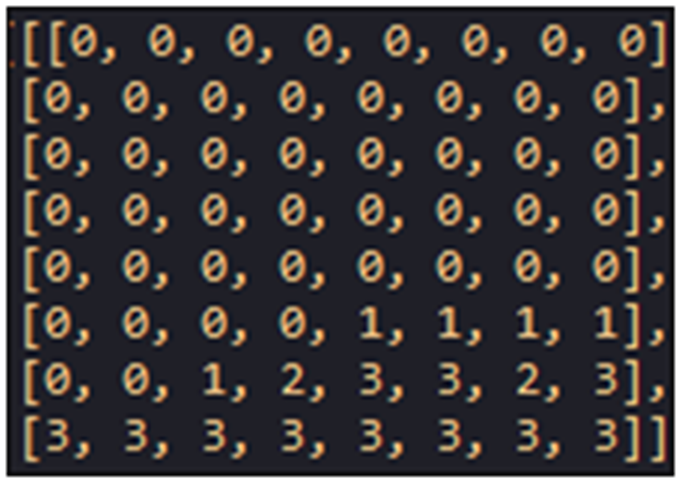}
    \caption{$K^{-3}(P)$}
    \label{fig:Moore neighboor}
    \end{subfigure}
\end{figure}
\newline 
From the process we observed that all the preimages of the pattern $P$ contain at least one of
the following patterns $33$ or $%
\begin{array}{c}
3 \\ 
3%
\end{array}%
.$
\newline
Now, Consider 
 $P=(%
\begin{array}{ccc}
3 & 3 & 3 \\ 
3 & 3 & 3 \\ 
3 & 3 & 3%
\end{array}%
)$ and $(P_{z},P_{e})=(3,3)$ then we will get the following results:%

\begin{figure}[h]
    \centering
    \begin{subfigure}[b]{0.49\textwidth}
    \centering
    \includegraphics[width=4cm]{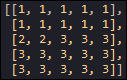}
    \caption{$K^{-1}(P)$}
    \label{fig:Moore neighboor}
    \end{subfigure}
    \hfill
    \begin{subfigure}[b]{0.49\textwidth}
    \centering
    \includegraphics[width=5cm,height=3.6cm]{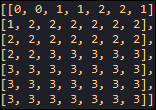}
    \caption{$K^{-2}(P)$}
    \label{fig:Moore neighboor}
    \end{subfigure}
\end{figure}
\begin{figure}[h]
    \centering
    \begin{subfigure}[b]{0.49\textwidth}
    \centering
    \includegraphics[width=4cm]{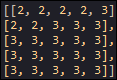}
    \caption{$K^{-1}(P)$}
    \label{fig:Moore neighboor}
    \end{subfigure}
    \hfill
    \begin{subfigure}[b]{0.49\textwidth}
    \centering
    \includegraphics[width=5cm,height=3.5cm]{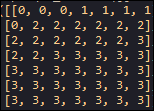}
    \caption{$K^{-2}(P)$}
    \label{fig:Moore neighboor}
    \end{subfigure}
\end{figure}
We remark that all the preimages of the pattern $P$ contain at least one of
the following patterns $333$,
$
\begin{array}{c}
3 \\ 
3 \\ 
3 \\
\end{array}
$
\end{example}

\begin{proposition}
\ \ \ \ \ \ \ \ \ \newline
The bassin of attraction of the fixed point $3^{\infty \times \infty }$
contain the cylinders $[33],
[\begin{array}{c}
3 \\ 
3  \\ 
\end{array}]$ for $(P_{z},P_{e})=(2,2)$
and contain the cylinders $[333]$,
$[\begin{array}{c}
     3 \\
     3 \\
     3 \\
\end{array}]$ for
$(P_{z},P_{e})=(3,3)$.
\end{proposition}

\begin{proof}
\ \ \ \ \ \ \ \ \ \ 

\begin{enumerate}
\item Consider $(P_{z},P_{e})=(2,2)$ and let $x$ be a point in $[33].$%
\newline
Without loss of generalities we suppose that $x_{-10}=x_{00}=3.$\newline
We have $K^{2}(x)_{[-1,1]\times \lbrack -2,1]}=3^{3\times 4}.$\newline
If we suppose that $K^{2n}(x)_{[-n,n]\times \lbrack
-2-n+1,n]}=3^{(2n+1)\times (2n+2)}$ then by induction we will have$\underset{%
n\rightarrow +\infty }{\lim }K^{2n}(x)=3^{\infty \times \infty }.$

\item Suppose that $(P_{z},P_{e})=(3,3)$ and consider a point $x$ in $%
[333]_{-10}.$ \newline
We have $K^{5}(x)_{[-2,2]\times \lbrack -2,2]}=3^{5\times 5}$ if we suppose
that \newline
$K^{5n}(x)_{[-n,n]\times \lbrack -n,n]}=3^{(2n+1)\times (2n+1)}$ then by
induction we will have\newline
$\underset{n\rightarrow +\infty }{\lim }K^{5n}(x)=3^{\infty \times \infty }.$%
\end{enumerate}
\end{proof}

\begin{proposition}
If $P_{z}$ or $P_{e}$ are greater than $4$ then the bassin of attraction of
the fixed point $3^{\infty \times \infty }$ do not contain any cylinder.
\end{proposition}

\begin{proof}
Let $[u]$ be a cylinder from $A^{%
\mathbb{Z}
^{2}}$ and note the size of $u$ by $(i,j)=size(u).$ We can suppose that $%
i=j>1.$ Then the sum of the coefficients of $2^{0}$ and $2^{1}$ of the cells
around the pattern $u$ have at most $3$ mutual cells with $u$.\newline
Or, as $P_{z}\geq 4$ or $P_{e}\geq 4$ then $\underset{n\rightarrow +\infty }{%
\lim }K^{n}(x)$ for a point $x\in \lbrack u]$ do not depend on the cylinder $%
[u].$
\end{proof}

\section{Conclusion}

We studied the topological properties of the two dimensional cellular
automaton model of language shift in Algeria for different values of the
parameters $P_{z}$ and $P_{e}$ where we showed that the cellular automaton
do not have any strictly temporally periodic points for any values of $P_{z}$
and $P_{e}$ and in particular for $(P_{z},P_{e})=(9,9)$ all points are fixed
points. Also, we proved that the cellular automaton $K$ is almost
equicontinuous and not surjective.

We developed an encoding for the bassin of attraction of the fixed points of
the cellular automaton from which we deduced that if one of the parameters
is greater than four then the bassin of attraction of the fixed point $%
3^{\infty \times \infty }$ do not contain any cylinder and we specified the
bassin of attraction of the same fixed point for particular values of $P_{z}$
and $P_{e}.$

\bibliographystyle{umthesis}
\bibliography{articles,book,texjourn,texnique,tubguide,tugboat,type,typeset,umthsmpl,unicode,uwthesis,ws-book-sample,ws-pro-sample,xampl}

\end{document}